\documentclass[12pt,reqno]{amsart}
\usepackage[british]{babel}
\usepackage{graphicx}
\usepackage{amsmath, amssymb, mathtools}
\usepackage{geometry, tikz, url, xcolor, adjustbox}
\usepackage{microtype}
\usetikzlibrary{arrows.meta, positioning, fit,backgrounds,calc}
 \usepackage{pdflscape}
\newcommand{\pqbinom}[3]{\binom{#1}{#2}_{\!#3}}

\newcommand{\ud}{\,\mathrm{d}}
\DeclareMathOperator{\Real}{Re}

\DeclareMathOperator{\Li}{Li}

\newtheorem{theorem}{Theorem}[section]
\newtheorem{proposition}[theorem]{Proposition}
\newtheorem{lemma}[theorem]{Lemma}
\newtheorem{corollary}[theorem]{Corollary}
\newtheorem*{theoremA}{Theorem A}
\theoremstyle{definition}

\theoremstyle{remark}

\numberwithin{equation}{section}

\title[Quartic limits for $(p,q)$-Rogers--Szeg\H{o}
  polynomials]{Quartic Fourier--Laplace limits for
  $(p,q)$-Rogers--Szeg\H{o} polynomials}

\begin{document}
\author{P. {\AA}hag} \address{Department
  of Mathematics and Mathematical Statistics, Ume{\aa} University,
  SE-901 87 Ume{\aa}, Sweden}
 \email{per.ahag@umu.se}

\author{R. Czy{\.z}}\address{Faculty of Mathematics and Computer
  Science, Jagiellonian University, \L ojasiewicza 6, 30-348 Krak\'ow,
  Poland}\email{rafal.czyz@im.uj.edu.pl}

\author{P. H. Lundow}
\address{Department of Mathematics and Mathematical Statistics,
  Ume{\aa} University, SE-901 87 Ume{\aa}, Sweden}
\email{per-hakan.lundow@umu.se}

\begin{abstract}
We prove asymptotic formulae for
$(p,q)$-Rogers--Szeg\H{o} polynomials with positive real parameters
when the leading quadratic term in the logarithmic expansion
about the middle coefficient vanishes. For coefficient indices
whose distance from $n/2$ is of order $n^{3/4}$, the coefficient ratios
have a quadratic--quartic exponential limit. A uniform bound valid
for all indices allows the ratios to be summed and gives locally
uniform convergence of the centred and normalised generating
polynomials to a Fourier--Laplace integral. We also obtain an asymptotic formula for
the sum of the coefficients, convergence of all rescaled moments,
weak convergence of the normalised coefficient measures, and
convergence with multiplicity of the zeros near $w=1$ for which
$n^{3/4}\log w$ remains bounded.
\end{abstract}



\maketitle

\section{Introduction}\label{sec:intro}

The $q$-binomial coefficients, or Gaussian polynomials,
\[
\pqbinom{n}{k}{q}
=
\prod_{j=1}^{k}\frac{1-q^{n-k+j}}{1-q^j},
\qquad 0\le k\le n,
\]
are standard objects in $q$-series and enumerative combinatorics.  They
specialise to the ordinary binomial coefficients at $q=1$, count subspaces
over finite fields when $q$ is a prime power, and occur as rank-generating
functions for partitions in a rectangle
\cite{GasperRahman2004,DLMF17,StanleyEC1}.  We use their homogeneous
symmetric two-parameter form
\begin{equation}\label{eq:intro-pq-def}
  \pqbinom{n}{k}{p,q}
  :=
  \prod_{j=1}^{k}\frac{p^{\,n-k+j}-q^{\,n-k+j}}{p^j-q^j},
  \qquad 0\le k\le n,
\end{equation}
with the usual limiting interpretation when $p=q$.  If $p\ne0$ and
$r=q/p$, then
\begin{equation}\label{eq:intro-pq-reduction}
  \pqbinom{n}{k}{p,q}
  =p^{k(n-k)}\pqbinom{n}{k}{r}.
\end{equation}
Thus the two parameters separate into an explicit quadratic factor and a
one-parameter Gaussian coefficient.

The associated generating polynomial
\[
P_n(w;p,q)
:=
\sum_{k=0}^{n}\pqbinom{n}{k}{p,q}w^k
\]
is the $(p,q)$-Rogers--Szeg\H{o} polynomial.  In the one-parameter
case, Rogers--Szeg\H{o} polynomials are connected with theta functions
and orthogonality on the unit circle in Szeg\H{o}'s classical work
\cite{Szego1926}; Lubinsky and Saff used them in the analysis of
Pad\'e approximants to partial theta functions and their zero
distribution \cite{LubinskySaff1987}.  Corcino developed algebraic
identities for the $(p,q)$-coefficients, including recurrences,
generating functions, inverse relations, and orthogonality relations
\cite{Corcino2008}, while Jagannathan and Sridhar considered the
corresponding Rogers--Szeg\H{o} polynomial in connection with a
$(p,q)$-oscillator \cite{JagannathanSridhar2010}.  Asymptotic formulae
for $q$-shifted factorials as $q\to1$ were obtained by McIntosh
\cite{McIntosh1999}; a $q$-binomial probability law with a continuous
Stieltjes--Wigert limit was studied by Kyriakoussis and Vamvakari
\cite{KyriakoussisVamvakari2013}.

For positive real $p$ and $q$, the coefficients in
\eqref{eq:intro-pq-def} are positive and symmetric in $k$, and
therefore become a probability law after normalisation.  Lundow and
Rosengren used these laws as finite-size approximations to Ising
magnetisation distributions
\cite{LundowRosengren2010,LundowRosengren2013}.  The complete graph is
represented exactly by the limiting case $p=q$, and the positive
coefficient array is always either unimodal or bimodal for $0 < p,q<1$
by the theorem of Su and Wang \cite{SuWang2012}.  The transition
function relating the parameters to the locations of the modes was
analysed in \cite{AhagCzyzLundow2024}.  The parameter choice below
makes the quadratic term in the expansion about the middle coefficient
vanish at leading order.  The theorem concerns the
$(p,q)$-Rogers--Szeg\H{o} polynomial itself and does not assert that a
finite Ising magnetisation law is exactly $(p,q)$-binomial.

Our earlier paper \cite{AhagCzyzLundowComplexAsymptotics} gives local
asymptotic formulae for coefficient-modulus ratios with complex parameters. Its Theorem~B includes the positive real quadratic--quartic estimate used
here. We give a direct proof of that special case. To sum the ratios over
$0\le k\le n$, an estimate outside each bounded range of the rescaled
index is also required. We prove such an estimate and then obtain the
locally uniform limit of the centred generating polynomial. The moment,
weak-convergence, and zero results follow from this limit.

For $z>0$ define
\begin{equation}\label{eq:intro-AB}
  A(z):=z\left(\coth\frac z4-1\right),
  \qquad
  B(z):=\frac{z^3\sinh(z/2)}{8\sinh^4(z/4)}.
\end{equation}
Then $B(z)>0$.  Fix $\gamma\in\mathbb R$, let $n$ tend to infinity through
even integers, and put
\begin{equation}\label{eq:intro-critical-params}
  \kappa:=\frac n2,
  \qquad
  u:=\frac{A(z)}2+\frac{\gamma}{\sqrt n},
  \qquad
  r:=e^{-z/n},
  \qquad
  p:=e^{-u/n},
  \qquad
  q:=p\, r.
\end{equation}
Write
\begin{equation}\label{def:cspz}
  C_{n,k}:=\pqbinom{n}{k}{p,q},
  \quad
  S_n:=\sum_{k=0}^nC_{n,k},
  \quad
  P_n(w):=\sum_{k=0}^{n}C_{n,k}w^k,
  \quad
  \mathcal Z_n(w):=w^{-\kappa}P_n(w).
\end{equation}
For the local coordinate $\zeta=n^{3/4}\log w$, set
\begin{equation}\label{eq:Gn-def}
  G_n(\zeta):=
  \frac{\mathcal Z_n(e^{\zeta/n^{3/4}})}
       {C_{n,\kappa}n^{3/4}},
  \qquad \zeta\in\mathbb C.
\end{equation}
Define the Fourier--Laplace transform
\begin{equation}\label{eq:intro-I}
  \mathcal I_{z,\gamma}(\zeta)
  :=
  \int_{-\infty}^{\infty}
  \exp\!\left(
  \zeta x + \gamma x^2
  -\frac{B(z)}{12}x^4
  \right)\ud x,
  \qquad \zeta\in\mathbb C.
\end{equation}

\begin{theoremA}
Under the preceding assumptions, $\mathcal I_{z,\gamma}$ is entire and
\begin{equation}\label{eq:intro-main-transform}
  G_n(\zeta)
       \longrightarrow
       \mathcal I_{z,\gamma}(\zeta)
\end{equation}
locally uniformly for $\zeta\in\mathbb C$ (see Fig.~\ref{figs} for an
illustration of this).  In particular, with $S_n:=\mathcal Z_n(1)$,
\begin{equation}\label{eq:intro-partition-limit}
  S_n=C_{n,\kappa}n^{3/4}\mathcal I_{z,\gamma}(0)(1+o(1))
\end{equation}
and
\begin{equation}\label{eq:intro-normalised-transform}
  \frac{G_n(\zeta)}{G_n(0)}
  \longrightarrow
  \frac{\mathcal I_{z,\gamma}(\zeta)}
       {\mathcal I_{z,\gamma}(0)}
\end{equation}
locally uniformly on $\mathbb C$.

Define probability measures $\nu_n$ by
\begin{equation}\label{eq:intro-nu-n}
  \nu_n:=
  \frac{1}{S_n}
  \sum_{\ell=-\kappa}^{\kappa}
  C_{n,\kappa+\ell}\,
  \delta_{\ell/n^{3/4}}.
\end{equation}
Here $\delta_x$ denotes the unit point mass at $x$. These measures
converge weakly to the probability measure
\begin{equation}\label{eq:intro-nu-limit}
  d\nu_{z,\gamma}(x)=
  \frac{1}{\mathcal I_{z,\gamma}(0)}
  \exp\!\left(
  \gamma x^2-\frac{B(z)}{12}x^4
  \right)\ud x.
\end{equation}
Moreover, if $U\subset\mathbb C$ is bounded and open and
$\partial U$ contains no zero of $\mathcal I_{z,\gamma}$, then, for all
sufficiently large even $n$, the functions
\[
\zeta\longmapsto\mathcal Z_n(e^{\zeta/n^{3/4}})
\quad\text{and}\quad
\zeta\longmapsto\mathcal I_{z,\gamma}(\zeta)
\]
have the same number of zeros in $U$, counted with multiplicity.  In
particular, if $s_0$ is a simple real zero of
$s\mapsto\mathcal I_{z,\gamma}(is)$, then there are real numbers
$s_n\to s_0$ such that
\[
P_n(e^{is_n/n^{3/4}})
=P_n(e^{-is_n/n^{3/4}})=0.
\]
\end{theoremA}

The locally uniform convergence in Theorem~A is the principal result.  It
is proved as Theorem~\ref{thm:generating-series-real}; the coefficient-sum
and weak-convergence statements are Theorem~\ref{thm:real-partition} and
Corollary~\ref{cor:fourier-laplace-real}, and the assertion about zeros is
Corollary~\ref{cor:nearby-zeros-real}.

The exponent $3/4$ follows from the cancellation of the quadratic term
when $2u=A(z)$.  If $\ell=xn^{3/4}$, then the perturbation
$u-A(z)/2=\gamma n^{-1/2}$ contributes $\gamma x^2$, whereas the
first non-zero term contributes $-B(z)x^4/12$.  At a type-$k$
Curie--Weiss critical point the fluctuation scale is $n^{1-1/(2k)}$;
for $k=2$ this is $n^{3/4}$ \cite{EllisNewman1978}.  Related
fluctuation and large-deviation results for mean-field block-spin
models appear in \cite{KnopfelLoweSchubertSinulis2020}.  The limiting
density in \eqref{eq:intro-nu-limit} is unimodal for $\gamma\le0$ and
bimodal for $\gamma>0$, with modes at $\pm\sqrt{6\gamma/B(z)}$ in the
latter case.

The final assertion of Theorem~A concerns zeros for which
$\zeta=n^{3/4}\log w$ remains bounded.  It follows from locally uniform
convergence and Rouch\'e's theorem.  {The limiting function has
only purely imaginary zeros by the classical result of Newman
\cite[Thm.~2]{Newman1976}, recalled before
Corollary~\ref{cor:nearby-zeros-real}.}  Lee--Yang theory studies zeros of
partition functions in a complex field \cite{LeeYang1952}.  Finite-size
Curie--Weiss and Ising zeros are also related to cumulants and large-deviation
statistics \cite{DegerFlindt2020,DegerBrangeFlindt2020}.  Kabluchko obtained
a global Lee--Yang zero distribution for the Curie--Weiss ferromagnet through
unitary Hermite polynomials and backward heat flow \cite{Kabluchko2025}.
Here we consider only zeros of the $(p,q)$-Rogers--Szeg\H{o} polynomial that
tend to $w=1$ at the rate stated in Theorem~A.

The paper is organised as follows. Section~\ref{sec:coefficient-estimate}
contains the elementary $(p,q)$ reduction and a direct proof of the positive
real coefficient estimate. Section~\ref{sec:generating-series-real}
proves the uniform quartic bound, the dominated Riemann-sum lemma, and the
Fourier--Laplace limit. It then gives the absolute normalisation, moment
convergence, weak convergence, and convergence of zeros near $w=1$.
Section~\ref{sec:conclusion} gives the conclusion and some open problems.

\section{Elementary reductions and coefficient estimate}\label{sec:coefficient-estimate}

We first prove the elementary identities needed below and then give a
direct proof of the near-central ratio estimate.

\begin{proposition}\label{prop:pq-reduction}
Let $p,q\in\mathbb C$ with $p\ne0$ and $p^j\ne q^j$ for $1\le j\le n$.
Set $r:=q/p$. Then, for every $0\le k\le n$,
\begin{equation}\label{eq:pq-factor}
  \pqbinom{n}{k}{p,q}=p^{\,k(n-k)}\pqbinom{n}{k}{r}.
\end{equation}
Moreover,
\begin{equation}\label{eq:pq-symmetry}
\pqbinom{n}{k}{p,q}=\pqbinom{n}{n-k}{p,q},
\qquad 0\le k\le n.
\end{equation}
Also,
\begin{equation}\label{eq:self-recip-polynomial}
  w^nP_n(w^{-1})=P_n(w),
\end{equation}
and the zeros of $P_n$ occur in reciprocal pairs, with multiplicity.
Finally, if $n$ is even then
\begin{equation}\label{eq:self-recip}
  \mathcal Z_n(w^{-1})=\mathcal Z_n(w).
\end{equation}
\end{proposition}

\begin{proof}
Since $r=q/p$, each factor in \eqref{eq:intro-pq-def} satisfies
\[
\frac{p^{\,n-k+j}-q^{\,n-k+j}}{p^j-q^j}
=
\frac{p^{\,n-k+j}(1-r^{n-k+j})}{p^j(1-r^j)}
=
p^{\,n-k}\frac{1-r^{n-k+j}}{1-r^j}.
\]
Multiplication over $j=1,\ldots,k$ gives \eqref{eq:pq-factor}.  Also
\[
\pqbinom{n}{k}{r}=\frac{(r;r)_n}{(r;r)_k(r;r)_{n-k}},
\qquad
(r;r)_m:=\prod_{j=1}^{m}(1-r^j),
\]
with $(r;r)_0=1$. Hence
$\pqbinom{n}{k}{r}=\pqbinom{n}{n-k}{r}$.  The factor $p^{k(n-k)}$ is
also invariant under $k\mapsto n-k$, and this proves \eqref{eq:pq-symmetry}.
The polynomial identity \eqref{eq:self-recip-polynomial} is exactly the same
coefficient symmetry.  Since the constant and leading coefficients of $P_n$
are both one, zero is not a zero of $P_n$; hence every zero is accompanied by
its reciprocal, with the same multiplicity.  Multiplying
\eqref{eq:self-recip-polynomial} by $w^{-n/2}$ gives \eqref{eq:self-recip}.
\end{proof}

The next theorem is the positive real specialisation of Theorem~B in
\cite{AhagCzyzLundowComplexAsymptotics}. The following proof uses only the
exact coefficient product and Taylor's formula.

\begin{theorem}\label{thm:coefficient-estimate}
Let $z>0$, let $\gamma\in\mathbb R$, and assume that $n$ is even. Put
$\kappa:=n/2$ and
\[
r:=e^{-z/n},
\qquad
u:=\frac{A(z)}2+\frac{\gamma}{\sqrt n},
\qquad
p:=e^{-u/n},
\qquad
q:=p\, r.
\]
For every fixed $L>0$, uniformly for all integer sequences $\ell_n$ such that
$x_n:=\ell_n/n^{3/4}$ satisfies $|x_n|\le L$,
\begin{equation}\label{eq:quartic-coefficient-estimate}
\frac{\pqbinom{n}{\kappa+\ell_n}{p,q}}
     {\pqbinom{n}{\kappa}{p,q}}
=
\exp\!\left(
\gamma x_n^2-\frac{B(z)}{12}x_n^4+O_{z,\gamma,L}(n^{-1/2})
\right).
\end{equation}
The same formula holds with $\kappa-\ell_n$ in place of $\kappa+\ell_n$.
See left panel of Fig.~\ref{figs} for an illustration of this ratio.
\end{theorem}

\begin{proof}
By \eqref{eq:pq-symmetry}, it is enough to consider integers
$0\le\ell\le Ln^{3/4}$. Put $h_z(s)=\log(1-e^{-zs})$ for $s>0$.
Multiplying adjacent ratios gives the exact identity
\[
D_n(\ell):=
\log\frac{\pqbinom{n}{\kappa+\ell}{r}}
         {\pqbinom{n}{\kappa}{r}}
=
\sum_{j=1}^{\ell}
\left[
 h_z\!\left(\frac12-\frac{j-1}{n}\right)
-h_z\!\left(\frac12+\frac jn\right)
\right].
\]
Let $a_n=1/2+1/(2n)$ and $d_j=(j-1/2)/n$. For all sufficiently large
$n$, the intervals $[a_n-d_j,a_n+d_j]$ lie in $[1/4,3/4]$, uniformly
in the stated range. Taylor's formula gives
\[
h_z(a_n-d_j)-h_z(a_n+d_j)
=-2h_z'(a_n)d_j-\frac{h_z^{(3)}(a_n)}{3}d_j^3+O_z(d_j^5).
\]
Since
\[
\sum_{j=1}^{\ell}\left(j-\frac12\right)=\frac{\ell^2}{2},
\qquad
\sum_{j=1}^{\ell}\left(j-\frac12\right)^3
=\frac{\ell^4}{4}-\frac{\ell^2}{8},
\]
we obtain
\[
D_n(\ell)
=-h_z'(a_n)\frac{\ell^2}{n}
-\frac{h_z^{(3)}(a_n)}{3n^3}
 \left(\frac{\ell^4}{4}-\frac{\ell^2}{8}\right)
+O_z\!\left(\frac{\ell^6}{n^5}\right).
\]
Now
\[
h_z'(1/2)=\frac{A(z)}2,
\qquad h_z^{(3)}(1/2)=B(z),
\qquad a_n-\frac12=\frac1{2n}.
\]
Replacing the derivatives at $a_n$ by those at $1/2$ therefore gives
\begin{equation}\label{eq:real-local-log-ratio}
D_n(\ell)
=-\frac{A(z)\ell^2}{2n}-\frac{B(z)\ell^4}{12n^3}
+O_{z,L}\!\left(\frac{\ell^2}{n^2}+\frac{\ell^6}{n^5}\right).
\end{equation}
Here the additional terms $O_z(\ell^4/n^4+\ell^2/n^3)$ are absorbed by
$O_z(\ell^2/n^2)$, since $\ell\le n$ for all sufficiently large $n$.
Finally, \eqref{eq:pq-factor} gives
\[
\log\frac{C_{n,\kappa+\ell}}{C_{n,\kappa}}
=D_n(\ell)+\frac{u\ell^2}{n}.
\]
Substitute $u=A(z)/2+\gamma n^{-1/2}$ and $\ell=x_n n^{3/4}$ in
\eqref{eq:real-local-log-ratio}. Both terms in the remainder are
$O_{z,L}(n^{-1/2})$ for $|x_n|\le L$, which proves
\eqref{eq:quartic-coefficient-estimate}. The case $\ell=0$ is exact,
and negative $\ell$ follow by symmetry.
\end{proof}

\section{Quartic summation and zeros near $w=1$}\label{sec:generating-series-real}

We first prove a uniform bound valid for all coefficient indices, including the
endpoints, and use it with Theorem~\ref{thm:coefficient-estimate} to
obtain the Fourier--Laplace limit. The asymptotic formula for the middle
coefficient is needed only for the absolute normalisation. Moments and
weak convergence follow from the transform limit, as does convergence
of zeros near $w=1$.

\begin{lemma}\label{lem:real-global-bound}
Let $z>0$, let $r=e^{-z/n}$, and assume that $n$ is even. Write
$\kappa=n/2$ and
\[
u:=\frac{A(z)}2+\gamma n^{-1/2},
\qquad
p=e^{-u/n}.
\]
Put
\[
g_n(k):=\log\!\left(p^{k(n-k)}\pqbinom{n}{k}{r}\right).
\]
Then there are constants $C,c>0$, depending only on $z$ and $\gamma$, such
that for all sufficiently large even $n$ and all $|\ell|\le\kappa$,
\begin{equation}\label{eq:quartic-global-bound}
g_n(\kappa+\ell)-g_n(\kappa)
\le C+|\gamma|\frac{\ell^2}{n^{3/2}}-c\frac{\ell^4}{n^3}.
\end{equation}
\end{lemma}

\begin{proof}
By symmetry, it is enough to consider $0\le \ell\le\kappa$. Put
\[
h_z(x):=\log(1-e^{-zx}),
\qquad 0<x\le1,
\qquad y:=\frac{\ell}{n}.
\]
The exact product formula gives
\[
D_n(\ell):=
\log\left(\frac{\pqbinom{n}{\kappa+\ell}{r}}
          {\pqbinom{n}{\kappa}{r}}\right)
=
\sum_{j=1}^{\ell}
\left[
 h_z\!\left(\frac{\kappa-\ell+j}{n}\right)
 -h_z\!\left(\frac{\kappa+j}{n}\right)
\right].
\]
The corresponding continuous exponent is
\[
G_z(y):=
\int_0^y\left(h_z(1/2-s)-h_z(1/2+s)\right)\,ds,
\qquad 0\le y\le 1/2.
\]
For $0<y<1/2$,
\[
h_z(1/2+y)-h_z(1/2-y)
=
\int_{-y}^{y}h_z'(1/2+s)\,ds,
\qquad
h_z'(x)=\frac{z}{e^{zx}-1}.
\]
Moreover,
\[
\frac{d^2}{dx^2}h_z'(x)=\frac{z^3e^{zx}(e^{zx}+1)}{(e^{zx}-1)^3}>0.
\]
Thus $h_z'$ is strictly convex, and Jensen's inequality gives
\[
h_z(1/2+y)-h_z(1/2-y)
\ge 2y h_z'(1/2)=A(z)y.
\]
After integration in $y$,
\[
Q_z(y):=G_z(y)+\frac{A(z)}{2}y^2\le0.
\]
The Taylor expansion at the origin gives
\[
Q_z(y)=-\frac{B(z)}{12}y^4+O_z(y^6),
\]
because $h_z'''(1/2)=B(z)$.  The strict convexity above gives
$Q_z(y)<0$ for $0<y\le1/2$, and $Q_z(y)/y^4$ extends continuously to
$y=0$ with value $-B(z)/12$. Hence there is $c_z>0$ such that
\begin{equation}\label{eq:G-quartic-bound}
G_z(y)+\frac{A(z)}{2}y^2\le -c_z y^4,
\qquad 0\le y\le1/2.
\end{equation}

We compare the sum with the integral. If $0\le y\le1/4$, all arguments of
$h_z$ stay in a compact subinterval of $(0,1]$, and Euler--Maclaurin gives
\[
D_n(\ell)=nG_z(y)+O_z(1).
\]
If $1/4\le y\le1/2$, write
\[
h_z(x)=\log(zx)+\eta_z(x),
\qquad
\eta_z(x)=\log(1-e^{-zx})-\log(zx).
\]
The function $\eta_z$ is smooth on $[0,1]$. Hence Euler--Maclaurin applies
uniformly to the $\eta_z$ part. For the logarithmic part one has the exact
identity
\[
\sum_{j=1}^{\ell}
\log\!\left(\frac{\kappa-\ell+j}{\kappa+j}\right)
=
\log\!\left(\frac{(\kappa!)^2}{(\kappa-\ell)!(\kappa+\ell)!}\right).
\]
The uniform log-factorial estimate
\[
\log m!=m\log m-m+O(\log(m+2)),\qquad m\ge0,
\]
with $0\log0:=0$, applied to $m=\kappa-\ell$, $m=\kappa$, and
$m=\kappa+\ell$, compares this expression with the corresponding integral
contribution to $nG_z(y)$ and gives an error $O_z(\log n)$.  In this range
\eqref{eq:G-quartic-bound} gives $Q_z(y)\le -\varepsilon_z$ for some
$\varepsilon_z>0$, so the logarithmic error is absorbed by the negative term of
order $n$.  Combining the two ranges, and changing the constants if
necessary, gives
\begin{equation}\label{eq:discrete-global-comparison}
D_n(\ell)+\frac{A(z)}{2}\frac{\ell^2}{n}
\le C_z-c_z\frac{\ell^4}{n^3},
\qquad 0\le\ell\le\kappa.
\end{equation}

Finally,
\[
g_n(\kappa+\ell)-g_n(\kappa)=D_n(\ell)+u\frac{\ell^2}{n},
\qquad
u=\frac{A(z)}2+\gamma n^{-1/2}.
\]
Together with \eqref{eq:discrete-global-comparison}, this gives
\[
g_n(\kappa+\ell)-g_n(\kappa)
\le C_z+\gamma\frac{\ell^2}{n^{3/2}}-c_z\frac{\ell^4}{n^3}
\le C_z+|\gamma|\frac{\ell^2}{n^{3/2}}-c_z\frac{\ell^4}{n^3}.
\]
The estimate for negative $\ell$ follows from
$\pqbinom{n}{\kappa+\ell}{r}=\pqbinom{n}{\kappa-\ell}{r}$.
\end{proof}

The Fourier--Laplace limit uses the following dominated Riemann-sum
lemma, including its parameter-dependent form.

\begin{lemma}\label{lem:dominated-riemann}
Let $h_n>0$ with $h_n\to0$ and let $I_n\subset\mathbb Z$ be intervals such that
$\sup I_n\,h_n\to\infty$ and $\inf I_n\,h_n\to-\infty$. Let
$H$ be a non-negative continuous integrable function on $\mathbb R$. Assume
that, for some $L_0>0$, both $x\mapsto H(x)$ and $x\mapsto H(-x)$ are
non-increasing on $[L_0,\infty)$. Assume that functions $F_n$ are
defined on the grid $\{j h_n:j\in I_n\}$ and that a continuous function
$F$ on $\mathbb R$ satisfies
\[
\sup_{\substack{j\in I_n\\ |j h_n|\le L}}|F_n(j h_n)-F(j h_n)|\to0
\]
for every fixed $L>0$, and
\[
|F_n(j h_n)|\le H(j h_n)
\qquad (j\in I_n).
\]
Then
\[
h_n\sum_{j\in I_n}F_n(j h_n)
\longrightarrow
\int_{-\infty}^{\infty}F(x)\ud x.
\]
Let $\Theta$ be a compact metric parameter space. If $F_{n,\theta}$ and
$F_\theta$ satisfy the preceding hypotheses with local convergence uniform in
$\theta\in\Theta$, if the same $H$ is valid for all $\theta\in\Theta$, and if
$(x,\theta)\mapsto F_\theta(x)$ is jointly continuous on
$\mathbb R\times\Theta$, then
\[
\sup_{\theta\in\Theta}
\left|
h_n\sum_{j\in I_n}F_{n,\theta}(j h_n)
-
\int_{-\infty}^{\infty}F_\theta(x)\ud x
\right|
\longrightarrow0.
\]
\end{lemma}

\begin{proof}
For fixed $L$, the local uniform convergence and the ordinary Riemann-sum
convergence give
\[
h_n\sum_{\substack{j\in I_n\\ |j h_n|\le L}}F_n(j h_n)
\longrightarrow
\int_{-L}^{L}F(x)\ud x.
\]
It remains to control the sum outside a fixed interval. By the assumptions on
$H$,
\[
\sup_{0<h\le1}h\sum_{|j| h>L}H(j h)
\longrightarrow0
\qquad (L\to\infty).
\]
Indeed, for $L$ beyond the monotonicity threshold, the right-hand sum is bounded
by $\int_{L-1}^{\infty}H(x)\ud x$, up to an endpoint term which is absorbed by
the same integral after increasing $L$; the left-hand sum follows by applying
the same argument to $x\mapsto H(-x)$. The local convergence and continuity
of $F$ imply $|F(x)|\le H(x)$ for every $x\in\mathbb R$, so the outer
integrals of $F$ are dominated by those of $H$. Letting first $n\to\infty$ and then $L\to\infty$
proves the first assertion.

For the parameter-dependent assertion, joint continuity makes the family
$\{F_\theta:\theta\in\Theta\}$ uniformly equicontinuous on every compact
interval. Hence the ordinary Riemann sums of $F_\theta$ over $[-L,L]$
converge uniformly in $\theta$. The local approximation by $F_{n,\theta}$ is
uniform by assumption. The same grid approximation gives
$|F_\theta(x)|\le H(x)$ for every $(x,\theta)\in\mathbb R\times\Theta$,
and the preceding tail estimate is uniform because the same $H$ applies to all
parameter values. This proves the second assertion.
\end{proof}

The local coefficient estimate and the uniform bound now give the
Fourier--Laplace limit stated in the introduction.

\begin{theorem}\label{thm:generating-series-real}
Let $z>0$, let $\gamma\in\mathbb R$ and assume that $n$ is even with
$\kappa:=n/2$ and the definitions of \eqref{def:cspz}. Let
\[
r:=e^{-z/n},
\qquad
u:=\frac{A(z)}{2}+\frac{\gamma}{\sqrt n},
\qquad
p:=e^{-u/n},
\qquad
q:=p\, r=e^{-(u+z)/n}.
\]
Let $G_n$ be given by \eqref{eq:Gn-def}, with $C_{n,k}$ and
$\mathcal Z_n$ formed from these parameters. For $\zeta\in\mathbb C$, set
\begin{equation}\label{eq:I-generating-def}
  \mathcal I_{z,\gamma}(\zeta)
  :=
  \int\limits_{-\infty}^{\infty}
  \exp\!\left(\zeta x + \gamma x^2-\frac{B(z)}{12}x^4\right)\ud x.
\end{equation}

Then $\mathcal I_{z,\gamma}$ is entire. Moreover, for every compact set
$K\subset\mathbb C$,
\begin{equation}\label{eq:generating-series-limit}
\sup_{\zeta\in K}
\left|
G_n(\zeta)
-
\mathcal I_{z,\gamma}(\zeta)
\right|
\longrightarrow 0
\qquad (n\to\infty,\ n\ \text{even}).
\end{equation}
Consequently,
\begin{equation}\label{eq:generating-series-normalised-limit}
\sup_{\zeta\in K}
\left|
\frac{G_n(\zeta)}{G_n(0)}
-
\frac{\mathcal I_{z,\gamma}(\zeta)}{\mathcal I_{z,\gamma}(0)}
\right|
\longrightarrow 0.
\end{equation}
\end{theorem}

\begin{proof}
The integral in \eqref{eq:I-generating-def} is absolutely convergent for
every $\zeta\in\mathbb C$, because $B(z)>0$ and the term
$-B(z)x^4/12$ dominates $\gamma x^2+\Real(\zeta)x$ as $|x|\to\infty$.
The same estimate, uniformly on compact subsets of the $\zeta$-plane,
shows that $\mathcal I_{z,\gamma}$ is entire.

The parameters $p$ and $q$ are positive real numbers. Hence all
coefficients $C_{n,k}$ are positive. Let $K\subset\mathbb C$ be compact,
and choose $M>0$ such that $|\zeta|\le M$ for $\zeta\in K$. Put
$x_\ell:=\ell/n^{3/4}$. Then
\[
G_n(\zeta)
=
\frac{1}{n^{3/4}}
\sum_{\ell=-\kappa}^{\kappa}
\frac{C_{n,\kappa+\ell}}{C_{n,\kappa}}e^{\zeta x_\ell}.
\]
For each fixed $L>0$, Theorem~\ref{thm:coefficient-estimate}
gives, uniformly for $|x_\ell|\le L$ and $\zeta\in K$,
\[
\frac{C_{n,\kappa+\ell}}{C_{n,\kappa}}e^{\zeta x_\ell}
=
\exp\!\left(
\zeta x_\ell + \gamma x_\ell^2-\frac{B(z)}{12}x_\ell^4
+O_{z,\gamma,L}(n^{-1/2})
\right).
\]
By Lemma~\ref{lem:real-global-bound},
\[
\frac{C_{n,\kappa+\ell}}{C_{n,\kappa}}
\le
C\exp\!\left(|\gamma|x_\ell^2-cx_\ell^4\right),
\qquad |\ell|\le\kappa.
\]
Consequently, uniformly for $\zeta\in K$,
\[
\left|
\frac{C_{n,\kappa+\ell}}{C_{n,\kappa}}e^{\zeta x_\ell}
\right|
\le
C\exp\!\left(M|x_\ell|+|\gamma|x_\ell^2-cx_\ell^4\right).
\]
The function on the right is integrable, and both it and its reflection are
non-increasing on a sufficiently far right-hand tail. The map
\[
(x,\zeta)\longmapsto
\exp\!\left(\zeta x + \gamma x^2-\frac{B(z)}{12}x^4\right)
\]
is jointly continuous on $\mathbb R\times K$, and hence uniformly continuous
on $[-L,L]\times K$ for every $L>0$. Thus all hypotheses of the uniform form
of Lemma~\ref{lem:dominated-riemann} hold with $h_n=n^{-3/4}$, and the
lemma gives \eqref{eq:generating-series-limit}.
Taking $\zeta=0$ in \eqref{eq:generating-series-limit} gives
\[
G_n(0)
\longrightarrow
\mathcal I_{z,\gamma}(0)>0,
\]
and division gives \eqref{eq:generating-series-normalised-limit}.
\end{proof}

The following estimate is a self-contained Euler--Maclaurin evaluation
of the central $q$-binomial coefficient for $r=e^{-z/n}$. It is not
used as a new general asymptotic theory for $q$-binomial
coefficients. Its role is to give the absolute normalisation in
Theorem~\ref{thm:real-partition}. The normalised generating-series
limit in Theorem~\ref{thm:generating-series-real} only uses ratios and
the uniform bound from Lemma~\ref{lem:real-global-bound}. General
asymptotic results for $q$-shifted factorials as $q\to1$ were given by
McIntosh \cite{McIntosh1999}.  For fixed $0<q<1$, Kyriakoussis and
Vamvakari proved pointwise convergence of a $q$-binomial distribution
to a Stieltjes--Wigert law \cite{KyriakoussisVamvakari2013}. Scaling
formulas for $q$-binomial and $p,q$-binomial coefficients were also
obtained (non-rigorously) in \cite{LundowRosengren2013} but for a
different parametrization ($r=1+z/n$).

\begin{lemma}\label{lem:central-coefficient-real}
Let $z>0$, let $r=e^{-z/n}$, and assume that $n$ is even with
$\kappa=n/2$. Define
\begin{equation}\label{eq:H-def}
\mathcal H(z):=
\frac{\Li_2(e^{-z})-2\Li_2(e^{-z/2})+\pi^2/6}{z}
\end{equation}
and
\begin{equation}\label{eq:d-def}
d(z):=\left(\frac{z(1-e^{-z})}{2\pi(1-e^{-z/2})^2}\right)^{1/2}.
\end{equation}
Then
\begin{equation}\label{eq:central-qbinom-asymp}
\pqbinom{n}{\kappa}{r}
=
d(z)n^{-1/2}\exp\!\left(n\mathcal H(z)\right)\left(1+O_z(n^{-1})\right).
\end{equation}
Consequently, if $|p|=e^{-u/n}$, then
\begin{equation}\label{eq:central-pq-asymp}
\left|\pqbinom{n}{\kappa}{p,p r}\right|
=
d(z)n^{-1/2}
\exp\!\left(n\mathcal H(z)-\frac{u n}{4}\right)
\left(1+O_z(n^{-1})\right).
\end{equation}
\end{lemma}
\begin{proof}
For $a=1/2$ or $a=1$, define
\[
L_n(a,z):=\sum_{j=1}^{an}\log(1-e^{-zj/n}).
\]
Since $n$ is even, both sums are over integers. With
$(r;r)_m:=\prod_{j=1}^{m}(1-r^j)$,
\[
\pqbinom{n}{\kappa}{r}
=\frac{(r;r)_n}{(r;r)_\kappa^2},
\qquad
\log\pqbinom{n}{\kappa}{r}=L_n(1,z)-2L_n(1/2,z).
\]
Put
\[
\eta_z(x):=\log(1-e^{-zx})-\log(zx).
\]
Then $\eta_z$ extends smoothly to $x=0$ and $\eta_z(0)=0$. For fixed
$a\in\{1/2,1\}$,
\[
L_n(a,z)
=
an\log z-an\log n+\sum_{j=1}^{an}\log j
+
\sum_{j=1}^{an}\eta_z(j/n).
\]
Stirling's formula gives
\[
\sum_{j=1}^{an}\log j
=
an\log(an)-an+\frac12\log(2\pi an)+O_z(n^{-1}).
\]
Euler--Maclaurin applied to the smooth function $\eta_z$ gives
\[
\sum_{j=1}^{an}\eta_z(j/n)
=
n\int_0^a\eta_z(x)\ud x+\frac{\eta_z(a)-\eta_z(0)}2+O_z(n^{-1}).
\]
Combining these two formulas and using $\eta_z(0)=0$ yields
\begin{equation}\label{eq:Ln-asymp}
L_n(a,z)
=
n\int_0^a\log(1-e^{-zx})\ud x
+\frac12\log\!\left(\frac{2\pi n}{z}(1-e^{-za})\right)
+O_z(n^{-1}).
\end{equation}
Indeed, the logarithmic constant is
\[
\frac12\log(2\pi an)+\frac12\eta_z(a)
=
\frac12\log\!\left(\frac{2\pi n}{z}(1-e^{-za})\right).
\]
Also,
\[
\int_0^a\log(1-e^{-zx})\ud x
=
\frac{\Li_2(e^{-za})-\Li_2(1)}{z},
\qquad \Li_2(1)=\frac{\pi^2}{6}.
\]
Inserting $a=1$ and $a=1/2$ in \eqref{eq:Ln-asymp} gives
\eqref{eq:central-qbinom-asymp}. Finally,
\[
\left|\pqbinom{n}{\kappa}{p,p r}\right|
=|p|^{\kappa^2}\pqbinom{n}{\kappa}{r}
=e^{-u n/4}\pqbinom{n}{\kappa}{r},
\]
which gives \eqref{eq:central-pq-asymp}.
\end{proof}

The sum of the coefficients and all rescaled moments follow from the
transform limit. The preceding lemma gives the absolute normalisation.

\begin{theorem}\label{thm:real-partition}
Let $z>0$, let $\gamma\in\mathbb R$, let $n$ be even, write
$\kappa=n/2$ and use the definitions of \eqref{def:cspz}. Put
\[
r=e^{-z/n},
\qquad
u:=\frac{A(z)}2+\frac{\gamma}{\sqrt n},
\qquad
p=e^{-u/n},
\qquad
q=p\, r=e^{-(u+z)/n},
\]
and define
\begin{equation}\label{eq:Im-def}
  \mu_m(z,\gamma):=
  \int\limits_{-\infty}^{\infty}
  x^m\exp\!\left(\gamma x^2-\frac{B(z)}{12}x^4\right)\ud x.
\end{equation}
Then
\begin{equation}\label{eq:S-quartic-central}
  S_n
  =
  C_{n,\kappa}\,n^{3/4} \mu_0(z,\gamma)(1+o(1)),
\end{equation}
and hence
\begin{equation}\label{eq:S-quartic-absolute}
  S_n
  =
  d(z)\mu_0(z,\gamma)n^{1/4}
  \exp\!\left(n\mathcal H(z)-\frac{u n}{4}\right)(1+o(1)).
\end{equation}
Moreover, for every integer $m\ge0$,
\begin{equation}\label{eq:quartic-moments}
  \frac{1}{S_n}\sum_{k=0}^{n}
  \left(\frac{k-\kappa}{n^{3/4}}\right)^m C_{n,k}
  \longrightarrow
  \frac{\mu_m(z,\gamma)}{\mu_0(z,\gamma)}.
\end{equation}
For odd $m$, the left-hand side is already zero by symmetry.
\end{theorem}

\begin{proof}
By Theorem~\ref{thm:generating-series-real},
\[
\frac{S_n}{C_{n,\kappa}n^{3/4}}
=G_n(0)\longrightarrow\mathcal I_{z,\gamma}(0)
=\mu_0(z,\gamma)>0.
\]
This proves \eqref{eq:S-quartic-central}, and
Lemma~\ref{lem:central-coefficient-real} gives
\eqref{eq:S-quartic-absolute}.

For every integer $m\ge0$, differentiation under the integral gives
$\mathcal I_{z,\gamma}^{(m)}(0)=\mu_m(z,\gamma)$. Locally uniform
convergence and Cauchy's integral formula imply
$G_n^{(m)}(0)\to\mu_m(z,\gamma)$. Since
\[
\frac{1}{S_n}\sum_{k=0}^{n}
\left(\frac{k-\kappa}{n^{3/4}}\right)^m C_{n,k}
=\frac{G_n^{(m)}(0)}{G_n(0)},
\]
division by $G_n(0)\to\mu_0(z,\gamma)>0$ proves
\eqref{eq:quartic-moments}. The odd moments vanish by coefficient symmetry.
\end{proof}

For positive real coefficients, the normalised generating series is the
Fourier--Laplace transform of a probability measure. The next corollary
gives the resulting weak convergence.

\begin{corollary}\label{cor:fourier-laplace-real}
Keep the assumptions of Theorem~\ref{thm:generating-series-real}. Define
probability measures on $\mathbb R$ by
\begin{equation}\label{eq:nu-n-def}
\nu_n:=
\frac{1}{\mathcal Z_n(1)}
\sum_{\ell=-\kappa}^{\kappa}C_{n,\kappa+\ell}\,
\delta_{\ell/n^{3/4}}.
\end{equation}
Here $\delta_x$ denotes the unit point mass at $x$. The limiting measure is
\begin{equation}\label{eq:nu-limit-def}
d\nu_{z,\gamma}(x):=
\frac{1}{\mathcal I_{z,\gamma}(0)}
\exp\!\left(\gamma x^2-\frac{B(z)}{12}x^4\right)\ud x.
\end{equation}
Here $\nu_n$ is a probability measure because $p,q>0$, and
$\nu_{z,\gamma}$ is a probability measure because $B(z)>0$ and
$\mathcal I_{z,\gamma}(0)$ is finite and positive. Then, for every compact
set $\Omega\subset\mathbb C$,
\begin{equation}\label{eq:fourier-laplace-measure-limit}
\sup_{\zeta\in \Omega}
\left|
\int_{\mathbb R}e^{\zeta x}\,d\nu_n(x)
-
\int_{\mathbb R}e^{\zeta x}\,d\nu_{z,\gamma}(x)
\right|
\longrightarrow0.
\end{equation}
In particular, for every $s\in\mathbb R$,
\begin{equation}\label{eq:characteristic-limit}
\int_{\mathbb R}e^{isx}\,d\nu_n(x)
\longrightarrow
\int_{\mathbb R}e^{isx}\,d\nu_{z,\gamma}(x).
\end{equation}
Consequently, $\nu_n$ converges weakly to $\nu_{z,\gamma}$.
Equivalently, if $K_n$ has law
$\mathbb P_n(K_n=k)=C_{n,k}/\mathcal Z_n(1)$, if $\mathbb E_n$ denotes
expectation with respect to $\mathbb P_n$, and if $M_n=n-2K_n$, then
\begin{equation}\label{eq:magnetisation-transform-limit}
\mathbb E_n\exp\!\left(\tau\frac{M_n}{n^{3/4}}\right)
\longrightarrow
\frac{\mathcal I_{z,\gamma}(-2\tau)}{\mathcal I_{z,\gamma}(0)}
\end{equation}
locally uniformly for $\tau\in\mathbb C$.
\end{corollary}

\begin{proof}
By the definition of $\nu_n$ and $\mathcal Z_n$,
\[
\int_{\mathbb R}e^{\zeta x}\,d\nu_n(x)
=
\frac{\mathcal Z_n(e^{\zeta/n^{3/4}})}{\mathcal Z_n(1)}
=
\frac{G_n(\zeta)}{G_n(0)}.
\]
By the definition of $\nu_{z,\gamma}$,
\[
\int_{\mathbb R}e^{\zeta x}\,d\nu_{z,\gamma}(x)
=
\frac{\mathcal I_{z,\gamma}(\zeta)}{\mathcal I_{z,\gamma}(0)}.
\]
Thus \eqref{eq:fourier-laplace-measure-limit} is exactly
\eqref{eq:generating-series-normalised-limit}. Taking $\zeta=is$ gives
\eqref{eq:characteristic-limit}. The limiting characteristic function is
continuous at the origin, and the weak convergence assertion follows from
the continuity theorem for characteristic functions \cite[Thm.~3.6.1]{Lukacs1970};
we use weak convergence in the standard sense of \cite[Ch.~1]{Billingsley1999}.
Finally, $M_n=-2(K_n-\kappa)$, so
\eqref{eq:magnetisation-transform-limit} follows from
\eqref{eq:fourier-laplace-measure-limit} with $\zeta=-2\tau$.
\end{proof}

The zeros of $\mathcal I_{z,\gamma}$ are purely imaginary. This is a
classical property of quartic Fourier transforms
\cite[Thm.~2]{Newman1976}. Indeed, put $a=B(z)/12>0$ and
\[
d\rho(x)=e^{-ax^4+(\gamma+1)x^2}\ud x.
\]
This even finite positive measure has the form (1.8) in
\cite{Newman1976}, with $m=0$ and an empty product. Taking $b=1$ in
that theorem shows that
\[
\int_{\mathbb R}e^{\zeta x-x^2}\,d\rho(x)
=\mathcal I_{z,\gamma}(\zeta)
\]
has only purely imaginary zeros. The next corollary concerns the
approximation of these zeros by those of the finite polynomials.

\begin{corollary}\label{cor:nearby-zeros-real}
  Keep the assumptions of Theorem~\ref{thm:generating-series-real} and recall
  the definitions of \eqref{def:cspz}. If $U\subset\mathbb C$ is a bounded open set whose boundary contains
  no zero of $\mathcal I_{z,\gamma}$, then, for all sufficiently large
  even $n$, the functions
  \[
  \zeta\longmapsto \mathcal Z_n(e^{\zeta/n^{3/4}})
  \quad\text{and}\quad
  \zeta\longmapsto \mathcal I_{z,\gamma}(\zeta)
  \]
  have the same number of zeros in $U$, counted with multiplicity.

  The local coordinate near $w=1$ is $\zeta=n^{3/4}\log w$; on the unit circle
  this becomes $w=e^{is/n^{3/4}}$. In particular, define
  \begin{equation}\label{eq:quartic-fourier-transform-real}
    \widehat{\nu}_{z,\gamma}(s)
    :=
    \frac{\mathcal I_{z,\gamma}(is)}{\mathcal I_{z,\gamma}(0)}
    =
    \frac{\displaystyle
      \int\limits_{-\infty}^{\infty}
      \exp\!\left(\gamma x^2-\frac{B(z)}{12}x^4\right)e^{isx}\ud x
    }{\displaystyle
      \int\limits_{-\infty}^{\infty}
      \exp\!\left(\gamma x^2-\frac{B(z)}{12}x^4\right)\ud x
    }.
  \end{equation}
  If $s_0\in\mathbb R$ is a simple zero of the real function
  $s\mapsto\widehat{\nu}_{z,\gamma}(s)$, then there are real numbers
  $s_n\to s_0$ such that
  \begin{equation}\label{eq:local-unit-zero-real}
    P_n(e^{is_n/n^{3/4}})=0.
  \end{equation}
  Consequently also $e^{-is_n/n^{3/4}}$ is a zero of
  $P_n$. Equivalently, a simple real zero of the limiting
  characteristic function produces a sequence of unit-circle zeros of
  the finite polynomials whose rescaled logarithms converge to $is_0$.
\end{corollary}

\begin{proof}
  By Theorem~\ref{thm:generating-series-real},
  \[
  G_n(\zeta)
  \longrightarrow
  \mathcal I_{z,\gamma}(\zeta)
  \]
  uniformly on compact subsets of $\mathbb C$. Since
  $C_{n,\kappa}n^{3/4}\ne0$, this scalar factor does not change
  zeros. Thus the zeros of $G_n$ are exactly the zeros of
  $\zeta\mapsto\mathcal Z_n(e^{\zeta/n^{3/4}})$. The zeros of
  $\mathcal I_{z,\gamma}$ in $\overline U$ are finite and, by
  assumption, lie in $U$. Choose pairwise disjoint closed discs
  contained in $U$, centred at these zeros, with boundaries containing
  no zero of $\mathcal I_{z,\gamma}$. By Rouch\'e's theorem, 
  $G_n$ and
  $\mathcal I_{z,\gamma}$ have the same number of zeros in each of
  these discs for all sufficiently large $n$. On the compact set
  obtained from $\overline U$ by removing the interiors of the discs,
  the function $\mathcal I_{z,\gamma}$ is bounded away from zero;
  uniform convergence gives that $G_n$ has no zero there for all
  sufficiently large $n$. This proves the first assertion.

  For the last assertion, note that $\widehat{\nu}_{z,\gamma}$ is
  real-valued on the real line, because the density in
  \eqref{eq:quartic-fourier-transform-real} is even. If $s_0$ is a
  simple real zero, then the real function
  $s\mapsto\widehat{\nu}_{z,\gamma}(s)$ changes sign in every
  sufficiently small interval around $s_0$. By
  \eqref{eq:generating-series-normalised-limit},
  \[
  s\mapsto
  \frac{\mathcal Z_n(e^{is/n^{3/4}})}{\mathcal Z_n(1)}
  \]
  converges uniformly to $\widehat{\nu}_{z,\gamma}(s)$ on compact real
  intervals.  The coefficient symmetry gives, for real $s$,
  \[
  \frac{\mathcal Z_n(e^{is/n^{3/4}})}{\mathcal Z_n(1)}
  =
  \frac{1}{\mathcal Z_n(1)}
  \left(
  C_{n,\kappa}
  +2\sum_{\ell=1}^{\kappa}C_{n,\kappa+\ell}
  \cos\!\left(\frac{s\ell}{n^{3/4}}\right)
  \right),
  \]
  and hence this function is real-valued. It therefore has a zero
  $s_n$ with $s_n\to s_0$. Since $\mathcal Z_n(w)=w^{-\kappa}P_n(w)$
  and $e^{is_n/n^{3/4}}\ne0$, this gives
  \eqref{eq:local-unit-zero-real}. The last statement follows from
  Proposition~\ref{prop:pq-reduction}.
\end{proof}

\begin{figure}[tbp]
  \begin{minipage}{0.48\textwidth}
    \centering
    \includegraphics[width=0.95\textwidth]{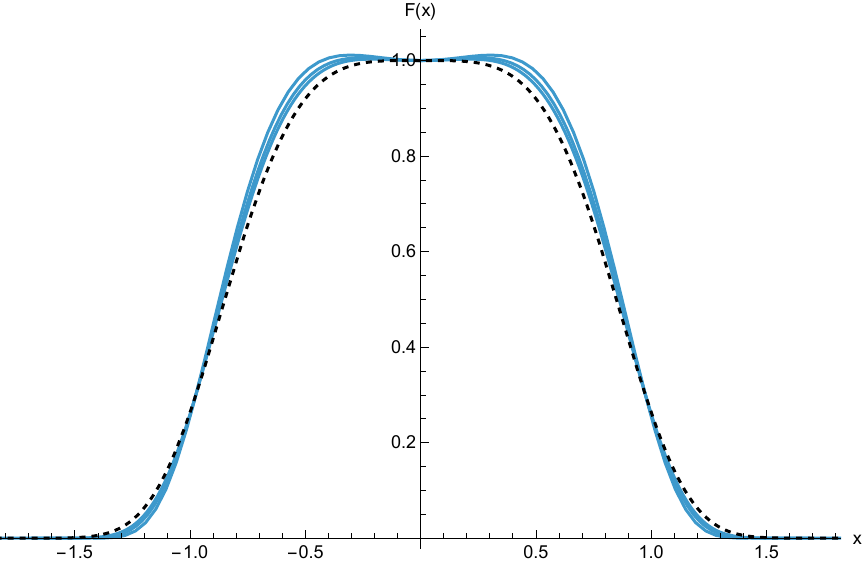}
  \end{minipage}%
  \begin{minipage}{0.48\textwidth}
    \centering
    \includegraphics[width=0.95\textwidth]{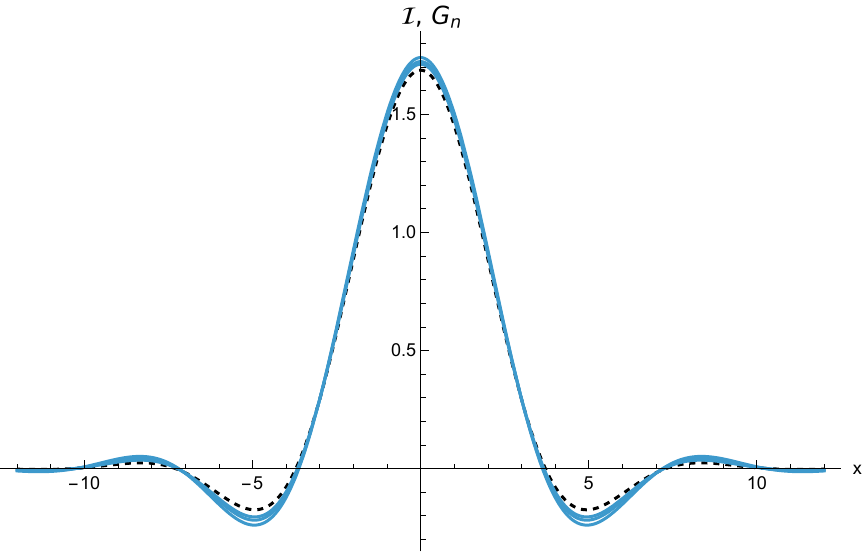}
  \end{minipage}
  \caption{Left: shows Eq.~\eqref{eq:quartic-coefficient-estimate} for
    $z=1$ and $\gamma=0$, ratios for $n=64,128,256$ (blue curves) and
    the limit $\exp(-1.333x^4)$ (dashed black curve).  Right: shows
    Eq.~\eqref{eq:intro-main-transform} on the imaginary axis,
    $G_n(ix)=\mathcal Z_n(e^{ix/n^{3/4}})/(C_{n,\kappa}n^{3/4})$ for
    $n=64,128,256$ (blue curves) and the limit
    $\mathcal I_{z,\gamma}(ix)$ (dashed black curve), for $z=1$ and
    $\gamma=0$.}
  \label{figs}
\end{figure}

\section{Conclusion and future work}\label{sec:conclusion}

We proved a quartic limit for centred $(p,q)$-Rogers--Szeg\H{o}
polynomials with positive real parameters.  The near-central
coefficient estimate applies to indices whose distance from $n/2$ is
of order $n^{3/4}$.  Together with the uniform quartic bound and
dominated Riemann sums, it gives locally uniform Fourier--Laplace
convergence.  We also obtained an asymptotic formula for the sum of
the coefficients, convergence of all rescaled moments, weak
convergence of the normalised coefficient measures, and convergence
with multiplicity of the zeros for which $\zeta=n^{3/4}\log w$ remains
bounded.

The purely imaginary location of the limiting zeros is classical
\cite{Newman1976}. A further problem is to obtain quantitative estimates
for their approximation by zeros of the finite polynomials. The convergence
proved here concerns bounded $\zeta=n^{3/4}\log w$. A description beyond
this neighbourhood requires estimates when $|\zeta|$ grows with $n$.

Another problem is to compare these zeros near $w=1$ with
Lee--Yang zeros of finite high-dimensional Ising magnetisation
polynomials.  Such a result would require proving, after the
corresponding normalisation, locally uniform convergence of the
centred Ising magnetisation polynomials at $w=e^{\zeta/n^{3/4}}$ to
the same Fourier--Laplace integral.  This could follow either from
direct comparison with $\mathcal Z_n$ or from coefficient estimates
and a uniform bound sufficient for the Riemann-sum argument.
Agreement of finitely many moments or fitted distributions does not by
itself imply zero convergence.

\section*{Declarations}

\subsection*{Funding}

The first-named author was partially funded by the Swedish Research Council under grant agreement no.~2025-05053 and by the Swedish Energy Agency under project no.~P2025-04323. The second-named author was partially
funded by the National Science Centre, Poland, under the Weave-UNISONO programme, grant no.~UMO-2025/07/Y/ST1/00146.

\subsection*{Competing interests}
The authors have no competing interests to declare that are relevant to the content of this article.

\subsection*{Author contributions}

All authors contributed equally to the mathematical development of the
work and to the preparation of the manuscript. 

\subsection*{Data availability}
No datasets were generated or analysed during the current study.

\end{document}